\documentclass[12pt]{amsart}

\usepackage{amssymb}

\usepackage{enumitem}

\usepackage{graphicx}

\makeatletter
\@namedef{subjclassname@2020}{%
  \textup{2020} Mathematics Subject Classification}
\makeatother

\usepackage[T1]{fontenc}

\newtheorem{theorem}{Theorem}[section]

\newtheorem{conjecture}[theorem]{Conjecture}
\newtheorem{lemma}[theorem]{Lemma}

\theoremstyle{definition}

\newtheorem{remark}[theorem]{Remark}

\numberwithin{equation}{section}

\allowdisplaybreaks

\begin{document}

\title[Arithmetic properties of DSOME function]{Arithmetic properties of DSOME function}


\author[N. D. Baruah]{Nayandeep Deka Baruah}
\address{Department of Mathematical Sciences, Tezpur University,\\  Assam 784028, India}
\email{nayan@tezu.ernet.in; nayandeeptezu@gmail.com}

\author[P. Gogoi]{Pankaj Gogoi}
\address{Department of Mathematical Sciences, Tezpur University,\\  Assam 784028, India}
\email{msp23110@tezu.ac.in; gopankajgo07@gmail.com}

\date{}

\begin{abstract}
Recently, Andrews and Ghosh Dastidar (Ramanujan J. \textbf{69}, Art. No. 26, 2026) studied two interesting functions $SOME(n)$ and $DSOME(n)$, where $SOME(n)$ is the sum of all the odd parts in the partitions of $n$ minus the sum of all the even parts and $DSOME(n)$ is the sum of all the odd parts in the partitions of $n$ into distinct parts minus the sum of all the even parts. They expressed the generating functions of $SOME(n)$ and $DSOME(n)$ in terms of $q$-series and found several interesting congruences modulo 4 and 5. In this paper, we express the generating function of $DSOME(n)$ in a closed form, which  allows us to find some new congruences and internal congruences modulo 4 and 8 for $DSOME(n)$. 
\end{abstract}




\maketitle

\section{Introduction}
For complex numbers $a$ and $q$ with $|q|<1,$ define the standard $q$-\emph{products}
$$(a;q)_{0}:=1,\, (a;q)_{n}:=\displaystyle\prod_{k=0}^{n-1}(1-aq^k),\, \text{and}\, (a;q)_{\infty}:=\lim_{n\to\infty}(a;q)_{n}.$$
\par
In the sequel, we set $f_n:=(q^n;q^n)_{\infty}$ for integers $n\geq1.$\\
A partition $\lambda:=\left(\lambda_1,\lambda_2,\dotsc,\lambda_k\right)$ of a positive integer $n$ is a finite non-increasing sequence of positive integers $\lambda_1,\lambda_2,\dotsc,\lambda_k $
such that $\lambda_1+\lambda_2+\dotsc+\lambda_k=n. $ The partition function $p(n)$ is defined as the number of partitions of $n.$ It is well known that the generating function of $p(n)$ is given by
$$\sum_{n=0}^{\infty}p(n)q^n=\frac{1}{f_1},$$
where by convention, we set $p(0)=1.$ The arithmetic properties of the partition function have been studied extensively after Ramanujan \cite{R5n} found his famous congruences modulo 5, 7 and 11, namely,
\begin{align*}
    p(5n+4)&\equiv 0\pmod{5},\\
     p(7n+5)&\equiv 0\pmod{7},\\
      p(11n+6)&\equiv 0\pmod{11}.
\end{align*}
\par
 Recently, Andrews and Ghosh Dastidar \cite{SOME} studied two interesting functions related to partitions, namely, $SOME(n)$ and $DSOME(n)$, where $SOME(n)$ is the sum of all the odd parts in the partitions of $n$  minus the sum of all the even parts and $DSOME(n)$ is the sum of all the odd parts in the partitions of $n$ into distinct parts minus the sum of all the even parts. They found that the generating functions of $SOME(n)$ and $DSOME(n)$ can be expressed as
 \begin{align} \sum_{n=0}^{\infty}SOME(n)q^n&=\frac{1}{(q;q)_{\infty}}\sum_{m\geq1}\frac{q^m}{\left(1+q^m\right)^2}\notag\\\intertext{and}
\label{DSOMEa}\sum_{n=0}^{\infty}DSOME(n)q^n&=(-q;q)_{\infty}\sum_{m\geq1}\frac{(-1)^{m-1}q^m}{\left(1+q^m\right)^2}\\
&=(-q;q)_{\infty}\sum_{n=1}^{\infty}\left(\frac{(2n-1)q^{2n-1}}{1+q^{2n-1}}-\sum_{n\geq0}\frac{(2n)q^{2n}}{1+q^{2n}}\right).\notag
 \end{align}
 They also proved the following beautiful congruences. For $n\geq0$, 
 \begin{align}
  SOME(5n+2)&\equiv 0\pmod{5},\notag\\
 SOME(5n+4)&\equiv 0\pmod{5},\notag\\
 SOME(4n)&\equiv 0\pmod{4},\notag\\
 DSOME(4n)&\equiv 0\pmod{4}.\notag
       \end{align}
    \par Moreover, they also conjectured that if $24\lambda\equiv1\pmod{5^\alpha}$, then
\begin{align*}
    SOME(\lambda)\equiv0\pmod{5^\alpha}.
\end{align*}
\par
In this paper, we express the generating function \eqref{DSOMEa} of $DSOME(n)$ in a closed form (see Theorem \ref{Theorem1.1} below). That allows us to find some new congruences and internal congruences for $DSOME(n)$ modulo 4 and 8. We present the main results in the following four theorems.

 \begin{theorem}\label{Theorem1.1}We have
 \begin{align}   \sum_{n\geq0}DSOME(n)q^n=\frac{1}{8}\bigg(\frac{f_2}{f_1}-\frac{f_1^7}{f_2^3}\bigg)\label{DSOMEgf}.
 \end{align} 
 \end{theorem}
\begin{theorem}\label{Theorem1.2}
For all $n\geq0$, we have
\begin{align}
DSOME(25n+5r+1)&\equiv0\pmod{4}, ~where~  1\leq r\leq4.\notag
\end{align}
\end{theorem} 
\begin{theorem}\label{Theorem1.3}
    For all $n\geq0$, we have
    \begin{align} DSOME(625n+125r+26)&\equiv0\pmod{8},~where~  1\leq r\leq4.\notag
\end{align}
\end{theorem}
\begin{theorem}\label{Theorem1.4}{For all $n\geq0$, we have}
      \begin{align*} DSOME(15625n+651)&\equiv 5DSOME(25n+1)+DSOME(625n+26)\\
  &\quad-5DSOME(n)\pmod{8}\\\intertext{and}
 DSOME(15625n+3125r+651)&\equiv 5DSOME(25n+5r+1)\\
&\quad+DSOME(625n+125r+26)\pmod{8},
\end{align*}
where  $1\leq r\leq4$.
\end{theorem}
We organize the paper in the following way. In Sect. \ref{sec2}, we present some preliminary definitions and lemmas that are useful for proving our theorems. In Sect. \ref{sec3}, we prove Theorems \ref{Theorem1.1}--\ref{Theorem1.4}. In the final section, we propose some conjectural congruences for $DSOME(n)$ modulo 8 and 16.

\section{Preliminary definitions and lemmas}\label{sec2}
The Rogers-Ramanujan identities are given by
\begin{align}
G(q):&=\sum_{n=0}^{\infty}\frac{q^{n^2}}{(q;q)_n}=\frac{1}{(q;\;q^5)_\infty(q^4;\;q^5)_\infty}\label{Gq}\\
\intertext{and}    H(q):&=\sum_{n=0}^{\infty}\frac{q^{n^2+n}}{(q;q)_n}=\frac{1}{(q^2;\;q^5)_\infty(q^3;\;q^5)_\infty},\label{Hq}\end{align}
where $G(q)$ and $H(q)$ are called the Rogers-Ramanujan functions. Recall from \cite{rogers} that \begin{align} R(q)=q^{1/5}\dfrac{H(q)}{G(q)},\notag\end{align}where $R(q)$ is the  famous Rogers-Ramanujan continued fraction defined by
\begin{align*}
 R(q):&=\frac{q^{1/5}}{1}\raisebox{-2.6mm}{\(+\)}\frac{q}{1}\raisebox{-2.6mm}{\(+\)}\frac{q^2}{1}\raisebox{-2.6mm}{\(+\cdots\)},~|q|<1. \end{align*}

Some useful results are presented in the following lemmas. For the first two lemmas, we set 
\begin{align}\label{TGH}T(q):=\frac{q^{1/5}}{R(q)}=\dfrac{G(q)}{H(q)}.\end{align}
\begin{lemma}\label{lemmaxy}\textnormal{(Baruah and Begum \cite{ND})}
    If $x:=T(q)$ and $y:=T(q^2),$ then
    \begin{align}
        xy^2-\frac{q^2}{xy^2}&=K,\label{Lemma3a}\\
        \frac{x^2}{y}-\frac{y}{x^2}&=\frac{4q}{K},\label{Lemma3b}\\
        \frac{y^3}{x}+q^2\frac{x}{y^3}&=K+\frac{4q^2}{K}-2q,\label{Lemma3c}\\
        x^3y+\frac{q^2}{x^3y}&=K+\frac{4q^2}{K}+2q,\label{Lemma3d}\\
        x^5-\dfrac{q^2}{x^5}&=\dfrac{f_1^6}{f_5^6}+11q,\label{x51byx5}
    \end{align}
    where   $K=f_2f_5^5/(f_1f_{10}^5).$
\end{lemma}
In the next lemma, we present two well-known 5-dissections of $f_1$ and $1/f_1$.
\begin{lemma}\textnormal{\cite[ p.~165]{spirit}} We have
    \begin{align}
       \label{c1} f_1&=f_{25}\left(T(q^5)-q-\frac{q^2}{T(q^5)}\right)
  \intertext{ and} 
        \label{c2}\frac{1}{f_1}&=\frac{f_{25}^5}{f_5^6}\bigg(T(q^5)^4+qT(q^5)^3+2q^2T(q^5)^2+3q^3T(q^5)+5q^4-\frac{3q^5}{T(q^5)}\\
        &\quad+\frac{2q^6}{T(q^5)^2}-\frac{q^7}{T(q^5)^3}+\frac{q^8}{T(q^5)^4}\bigg)\notag.
    \end{align}
\end{lemma}

The following gives 2-dissections of $G(q)$ and $H(q)$ due to Watson \cite[p.~60]{watson}.
\begin{lemma}We have
\begin{align}
    G(q)=\frac{f_8}{f_2}\left(G\left(q^{16}\right)+qH(-q^4)\right)\label{c5}\\\intertext{and}
    H(q)=\frac{f_8}{f_2}\left(G(-q^{4})+q^3H(q^{16})\right)\label{watson}.
\end{align}
    \end{lemma}
   
    We use the identities in the following lemma quite frequently in the next sections.
\begin{lemma}\label{lemmabb}
\textnormal{(Baruah and Begum \cite[Eqs. (2.6) and (2.7)]{ND})}
    \begin{align}
        \frac{f_5^5}{f_1^4f_{10}^3}&=\frac{f_5}{f_2^2f_{10}}+4q\frac{f_{10}^2}{f_1^3f_2},\label{Lemma4a}\\
        \frac{f_2^3f_5^2}{f_1^2f_{10}^2}&=\frac{f_5^5}{f_1f_{10}^3}+q\frac{f_{10}^2}{f_2}.\label{Lemma4b}
     \end{align}
\end{lemma}

We conclude this section by proving another useful result.
\begin{lemma} We have
    \begin{align}\label{lemmaG}
        G(q)G(q^4)\pm qH(q)H(q^4)\equiv \frac{1}{f_2}\pmod{2}.
    \end{align}
\end{lemma}
\begin{proof}
 One of Ramanujan's famous forty identities for $G(q)$ and $H(q)$ (See \cite{bcbetal} for more details) is given by
    \begin{align}
        G(q)G(q^4)-qH(q)H(q^4)=\frac{f_{10}^5}{f_2f_5^2f_{20}^2}.\label{ggq}
    \end{align}
    However, by the binomial theorem, for positive integers $k$ and $\ell$, it can be easily shown that
\begin{align}
    f_{k}^{2^\ell}\equiv f_{2k}^{2^{\ell-1}}\pmod{2^\ell}.\label{binom}
\end{align}
 In particular,
 \begin{align*}
    f_1^2\equiv f_2\pmod{2}~\textup{and}~f_5^2\equiv f_{10}\pmod{2}.
\end{align*}
Thus, \eqref{lemmaG} readily follows from \eqref{ggq}.
\end{proof}

\begin{remark}
 We will frequently use \eqref{binom} without further reference to it.   
\end{remark} 

    \section{Proofs of Theorems \ref{Theorem1.1} --\ref{Theorem1.4}}\label{sec3}
  \begin{proof}[ Proof of Theorem \ref{Theorem1.1}.]
 We recall from \cite[p.~61, Eq. (3.3.6)]{spirit} that
 \begin{align}
        \dfrac{(q;q)_\infty^4}{(-q;q)_\infty^4}=1-8\sum_{n=1}^{\infty}\left(\frac{(2n-1)q^{2n-1}}{1+q^{2n-1}}-\sum_{n\geq0}\frac{2nq^{2n}}{1+q^{2n}}\right).\notag
  \end{align}
  Employing the above identity in \eqref{DSOMEa}, we have
  \begin{align*}
    \sum_{n\geq0}DSOME(n)q^n&=(-q;q)_\infty\bigg( \sum_{n\geq0}\frac{(2n-1)q^{2n-1}}{1+q^{2n-1}}-\sum_{n\geq0}\frac{2nq^{2n}}{1+q^{2n}}\bigg)\\
    &=\frac{1}{8}(-q,q)_\infty\bigg(1-\dfrac{(q;q)_\infty^4}{(-q;q)_\infty^4}\bigg)\\
    &=\frac{1}{8}\frac{f_2}{f_1}\bigg(1-\frac{f_1^8}{f_2^4}\bigg)\\
    &=\frac{1}{8}\bigg(\frac{f_2}{f_1}-\frac{f_1^7}{f_2^3}\bigg),  
\end{align*}
which is \eqref{DSOMEgf}.
 \end{proof}
\begin{proof}[Proof of Theorem \ref{Theorem1.2}] From  \eqref{DSOMEgf} and \eqref{Lemma4a}, 
 \begin{align*}
\sum_{n=0}^{\infty}DSOME(n)q^n
&=\frac{f_2}{8f_1}-\frac{1}{8}\left(\frac{f_1^3f_5^4}{f_2f_{10}^2}-4q\frac{f_1^4f_{10}^3}{f_2^2f_5}\right)\notag\\
&=\frac{f_2}{8f_1}-\frac{1}{8}\left(\frac{f_{2}f_{5}^{8}}{f_{1}f_{10}^{4}}-4qf_{5}^{3}f_{10}\right)+\frac{q}{2}\left(f_{5}^{3}f_{10}-4q\frac{f_{1}f_{10}^{6}}{f_{2}f_{5}^{2}}\right)\notag\\
&=\frac{f_2}{8f_1}-\frac{f_{2}f_{5}^{8}}{8f_{1}f_{10}^{4}}+qf_{5}^{3}f_{10}-2q^2\frac{f_{10}^{6}f_{1}}{f_{5}^{2}f_{2}},\end{align*}
which, by \eqref{c1} and \eqref{c2}, can be recast as
\begin{align}
\label{c67}&\sum_{n=0}^{\infty}DSOME(n)q^n\\
&=\frac{f_{25}^{5}f_{50}}{f_{5}^{6}}\left(\frac{1}{8}-\frac{f_{5}^{8}}{8f_{10}^{4}}\right)\left(T(q^{10})-q^2-\frac{q^4}{T(q^{10})}\right)\bigg(T^4(q^5)+qT^3(q^5)
\notag\\
&\quad+2qT^2(q^5)+3q^3T(q^5)+5q^4-3\frac{q^5}{T(q^5)}+2\frac{q^6}{T^2(q^5)}
-\frac{q^7}{T^3(q^5)}
+\frac{q^8}{T^4(q^5)}\bigg)\notag\\
&\quad+qf_{5}^{3}f_{10}-2q^2\frac{f_{25}f_{50}^{5}}{f_{5}^{2}}\left(T(q^5)-q-\frac{q^2}{T(q^5)}\right)\bigg(T^4(q^{10})+qT^3(q^{10})\notag\\
&\quad+2q^{2}T^2(q^{10})+3q^{6}T(q^{10})+5q^8-3\frac{q^{10}}{T(q^{10})}+
2\frac{q^{12}}{T^2(q^{10})}-\frac{q^{14}}{T^3(q^{10})}\notag\\
&\quad+\frac{q^{16}}{T^4(q^{10})}\bigg).\notag
\end{align}
Extracting the terms involving $q^{5n+1}$ from both sides of \eqref{c67}, dividing by $q$ and then replacing $q^5$ by $q$, we find that

\begin{align*}
&\sum_{n=0}^{\infty}DSOME(5n+1)q^n\notag\\
&=f_{1}^{3}f_{2}+\left(\frac{1}{8}-\frac{f_{1}^{8}}{8f_{2}^{4}}\right)\frac{f_{5}^{5}f_{10}}{f_{1}^{6}}\bigg(x^{3}y+\frac{q^2}{x^{3}y}-2q\left(\frac{x^2}{y}-\frac{y}{x^2}\right)
-5q\bigg)\\
&\quad-2\frac{f_{5}f_{10}^{5}}{f_{1}^{2}}\bigg(2q\left(xy^2-\frac{q^2}{xy^2}\right)-q\left(\frac{y^3}{x}+q^2\frac{x}{y^3}\right)
-5q^2\bigg),\end{align*}
which by \eqref{Lemma3a} - \eqref{Lemma3d} can be recast as
\begin{align}
\label{e1}&\sum_{n=0}^{\infty}DSOME(5n+1)q^n\\
&=f_1^3f_2+\frac{1}{8}\bigg(1-\frac{f_{1}^{8}}{8f_{2}^{4}}\bigg)\frac{f_{5}^{5}f_{10}}{f_{1}^{6}}\bigg(K-\frac{4q^2}{K}-3q\bigg)-2q\frac{f_{5}f_{10}^{5}}{f_{1}^{2}}
\bigg(K-\frac{4q^2}{K}-3q\bigg)\notag.
    \end{align}

Now, by \eqref{Lemma4a} and \eqref{Lemma4b}, we have
    \begin{align}
       \label{qq12} K-\frac{4q^2}{K}-3q&=\left(K+q\right)\left(1-\frac{4q}{K}\right)=\left(\frac{f_{2}f_{5}^{5}}{f_{1}f_{10}^{5}}+q\right)\left(1-4q\frac{f_{1}f_{10}^{5}}{f_{2}f_{5}^{5}}\right)\\
    &=\frac{f_{2}^{4}f_{5}^{2}}{f_{1}^{2}f_{10}^{4}}\times\frac{f_{1}^{4}f_{10}^{2}}{f_{2}^{2}f_{5}^{4}}=\frac{f_{1}^{2}f_{2}^{2}}{f_{5}^{2}f_{10}^{2}}.\notag
    \end{align}
  Employing \eqref{qq12} in \eqref{e1}, we have
    \begin{align}
&\sum_{n=0}^{\infty}DSOME(5n+1)q^n=f_{1}^{3}f_{2}+\frac{1}{8}\frac{f_{2}^{2}f_{5}^{3}}{f_{1}^{4}f_{10}}-\frac{1}{8}\frac{f_{1}^{4}f_{5}^{3}}{f_{2}^{2}f_{10}}-2q\frac{f_{2}^{2}f_{10}^{3}}{f_{5}}.\end{align}
Applying \eqref{Lemma4a} in the above, we find that
\begin{align}\label{c8}&\sum_{n=0}^{\infty}DSOME(5n+1)q^n\\
&=f_{1}^{3}f_{2}+\frac{1}{8}\left(\frac{f_{10}}{f_{5}}+4q\frac{f_{2}f_{10}^{4}}{f_{1}^{3}f_{5}^{2}}\right)-\frac{1}{8}\left(\frac{f_{5}^{7}}{f_{10}^{3}}-4q\frac{f_{1}f_{5}^{2}f_{10}^{2}}{f_{2}}\right)-2q\frac{f_{2}^{2}f_{10}^{3}}{f_{5}}\notag\\
&=f_{1}^{3}f_{2}+\frac{1}{8}\frac{f_{10}}{f_5}+\frac{q}{2}\left(\frac{f_{1}f_{10}^{6}}{f_{2}f_{5}^{6}}+4q\frac{f_{10}^{9}}{f_{1}^{2}f_{5}^{7}}\right)-\frac{f_{5}^{7}}{8f_{10}^{3}}+\dfrac{q}{2}\frac{f_{1}f_{5}^{2}f_{10}^{2}}{f_{2}}-2q\frac{f_{2}^{2}f_{10}^{3}}{f_{5}}\notag\\
&=\frac{f_{10}}{8f_5}-\frac{f_5^7}{8f^3_{10}}+f_1^3f_2
+\dfrac{q}{2}\left(\frac{f_{1}f_{10}^6}{f_2f_5^6}+\frac{f_{1}f_5^2f_{10}^2}{f_2}\right)-2q\frac{f_{2}^{2}f_{10}^{3}}{f_{5}}
+2q^2\frac{f_{10}^{9}}{f_{1}^{2}f_{5}^{7}}.\notag
\end{align}
Therefore, 
\begin{align*}&\sum_{n=0}^{\infty}DSOME(5n+1)q^n\\
&\equiv\frac{f_{10}}{8f_5}-\frac{f_5^7}{8f^3_{10}}+f_1^3f_2
+q\frac{f_{1}f_5^2f_{10}^2}{f_2}-2q\frac{f_{2}^{2}f_{10}^{3}}{f_{5}}
+2q^2\frac{f_{10}^{9}}{f_{1}^{2}f_{5}^{7}}\pmod{4},
\end{align*}
which by \eqref{Lemma4b} yields
\begin{align*}&\sum_{n=0}^{\infty}DSOME(5n+1)q^n\\
&=\frac{f_{10}}{8f_5}-\frac{f_5^7}{8f^3_{10}}+f_1^3f_2
+\left(\frac{f_2^3f_5^4}{f_1f_{10}^2}-\frac{f_5^7}{f_{10}^3}\right)-2q\frac{f_{2}^{2}f_{10}^{3}}{f_{5}}
+2q^2\frac{f_{10}^{9}}{f_{1}^{2}f_{5}^{7}}\notag\\
&\equiv \frac{f_{10}}{8f_5}-\frac{9}{8}\frac{f_5^7}{f_{10}^3}+2f_1f_4+2q\frac{f_4f_{10}^3}{f_5}+2q^2\frac{f_5f_{10}^5}{f_2}\pmod{4}.
\end{align*}
 Employing \eqref{c1} and \eqref{c2} in the above, we have
       \begin{align}              
    \label{c7} &\sum_{n=0}^{\infty}DSOME(5n+1)q^n\\
    &\equiv \frac{f_{10}}{8f_5}-\frac{9}{8}\frac{f_5^7}{f_{10}^3}+2f_{25}f_{100}\bigg(T(q^5)-q-\frac{q^2}{T(q^5)}\bigg)\bigg(T(q^{20})-q^4-\frac{q^8}{T(q^{20})}\bigg)\notag\\
    &\quad+2q\frac{f_{10}^3f_{100}}{f_5}
  \bigg(T(q^{20})-q^4-\frac{q^8}{T(q^{20})}\bigg)
         +2q^2\frac{f_5f_{10}^5f_{50}^5}{f_{10}^6}\bigg(T^4(q^{10})\notag\\
         &\quad+q^2T^3(q^{10})+2q^4T^2(q^{10})+3q^6T(q^{10})+5q^8-3\frac{q^{10}}{T(q^{10})}
         +2\frac{q^{12}}{T^2(q^{10})}\notag\\
         &\quad-\frac{q^{14}}{T^3(q^{10})}+\frac{q^{16}}{T^4(q^{10})}\bigg)\pmod{4}\notag.
       \end{align}
       
       To derive Theorem \ref{Theorem1.2}, we need to extract the terms involving $q^{5n+r}$ for $r=1,2,3$, and $4$ in  \eqref{c7}. We treat each of the cases for $r$ separately in the following. 

      \vskip .4cm \noindent \emph{Case $r=1$}. Extracting the terms involving $q^{5n+1}$ from both sides of \eqref{c7}, dividing by $q$, and then replacing $q^5$ by $q$, we obtain
      \begin{align*} &\sum_{n=0}^{\infty}DSOME(25n+6)q^n\\
      &\equiv 2f_5f_{20}\bigg(-T(q^4)+\frac{q}{T(q)}\bigg)+2\frac{f_2^3f_{20}T(q^4)}{f_1}+2q^3\frac{f_1f_{10}^5}{f_2T^3(q^2)}\pmod{4},\end{align*}
      which by \eqref{TGH} can be rewritten as
      \begin{align}\label{ty5}&\sum_{n=0}^{\infty}DSOME(25n+6)q^n\\
      &\equiv 2f_5f_{20}\bigg(\frac{G(q^4)}{H(q^4)}-q\frac{H(q)}{G(q)}\bigg)+2f_1f_4f_{20}\frac{G(q^4)}{H(q^4)}+2q^3\frac{f_1f_{10}^5H^3(q^2)}{f_2G^3(q^2)}\notag\\
         &\equiv 2\frac{f_5f_{20}}{G(q)H(q^4)}\bigg(G(q)G(q^4)-qH(q)H(q^4)\bigg)+2f_1f_4f_{20}\frac{G(q^4)}{H(q^4)}\notag\\&\quad+2q^3\frac{f_1f_{10}f_{40}H(q^2)H(q^4)}{f_2G(q^2)G(q^4)}\pmod{4},\notag
    \end{align}
    where we also used the fact that for any positive integer $j$,
    \begin{align}\label{HGmod2}G^2(q^j)\equiv G(q^{2j})\pmod{2}\quad\textup{and}\quad H^2(q^j)\equiv H(q^{2j})\pmod{2}.
    \end{align}

    Next, it follows from \eqref{Gq} and \eqref{Hq} that
     \begin{align} H(q)G(q)&=\frac{f_{5}}{f_1}.\label{HqGq}
\end{align}
Employing \eqref{lemmaG}, \eqref{HGmod2}, and \eqref{HqGq} in \eqref{ty5}, we find that
 \begin{align}   \label{dsome25n6}&\sum_{n=0}^{\infty}DSOME(25n+6)q^n\\
 &\equiv 2\frac{f_5f_{20}}{f_2G(q)H(q^4)}+2f_1f_8G^2(q^4)+2q^3f_1f_{40}\dfrac{H(q^8)}{G(q^4)}\notag\\ 
&\equiv 2f_1f_2H(q)G(q^4)+2f_1f_8G(q^4)\left(G(q^4)+q^3H(q^{16})\right)\pmod{4}.\notag
\end{align}

However, from \eqref{watson}, we have
       \begin{align} H(q)&\equiv \frac{f_8}{f_2}\left(G(q^4)+q^3H(q^{16})\right)\pmod{2}.\label{H(1)}
        \end{align}
Employing \eqref{H(1)} in \eqref{dsome25n6}, we find that
    \begin{align}
\sum_{n=0}^{\infty}DSOME(25n+6)q^n&\equiv 2f_1f_2H(q)G(q^4)+2f_1f_2H(q)G(q^4)\equiv0\pmod{4}.\notag 
        \end{align}
        This completes the proof of the case $r=1$.
        
      \vskip .4cm \noindent \emph{Case $r=2$}.
  Extracting the terms involving $q^{5n+2}$\ from both sides of \eqref{c7}, dividing by $q^2$, and then replacing $q^5$ by $q$, we obtain
  \begin{align*}
\sum_{n=0}^{\infty}DSOME(25n+11)q^n
    &\equiv 2f_5f_{20}\frac{T(q^4)}{T(q)}+2\frac{f_1f_{10}^5}{f_2}\left(T^4(q^2)-\frac{3q^2}{T(q^2)}\right)\pmod{4}.
    \end{align*}
    Using \eqref{TGH}, \eqref{HGmod2}, and \eqref{HqGq}, we find that
  \begin{align*}
&\sum_{n=0}^{\infty}DSOME(25n+11)q^n\\
    &\equiv 2f_5f_{20}\frac{G(q^4)H(q)}{H(q^4)(q)}+2\frac{f_1f_{10}^5}{f_2}\left(\frac{G^4(q^2)}{H^4(q^2)}+q^2\frac{H(q^2)}{G(q^2)}\right)\notag\\    &\equiv 2f_1f_4H^2(q)G^2(q^4)+2\frac{f_1f_{10}^5}{f_2}\left(\frac{G(q^8)}{H(q^8)}+q^2\frac{H(q^2)}{G(q^2)}\right)\notag\\
      &\equiv 2f_1f_4H(q^2)G(q^8)+2\frac{f_1f_{10}^5}{f_2G(q^2)H(q^8)}\left(G(q^2)G(q^8)+q^2H(q^2)H(q^8)\right)\pmod{4}.
      \end{align*}
Using \eqref{lemmaG} in the above and then employing \eqref{HqGq}, we obtain
   \begin{align}  \sum_{n=0}^{\infty}DSOME(25n+11)q^n &\equiv 2f_1f_4H(q^2)G(q^8)+2\frac{f_1f_{10}^5}{f_2f_4G(q^2)H(q^8)}\notag\\
      &\equiv 2f_1f_4H(q^2)G(q^8)+2\frac{f_1f_4f_{10}f_{40}}{f_2f_8G(q^2)H(q^8)}\notag\\ &\equiv 2f_1f_4H(q^2)G(q^8)+ 2f_1f_4H(q^2)G(q^8)\notag\\
      &\equiv  4f_1f_4H(q^2)G(q^8)\notag\\
      &\equiv 0\pmod{4},\notag
      \end{align}
which completes the proof of the case $r=2$.

      \vskip .4cm \noindent \emph{Case $r=3$}.
We extract the terms involving $q^{5n+3}$ from both sides of \eqref{c7} and then proceed as in the previous two cases. Accordingly, we deduce that
\begin{align*}
&\sum_{n=0}^{\infty}DSOME(25n+16)q^n\notag\\
&\equiv 2qf_5f_{20}\frac{T(q)}{T(q^4)}+2q\frac{f_1f_{10}^5}{f_2}\left(T(q^2)+\frac{q^2}{T^4(q^2)}\right)\notag\\
    &\equiv 2qf_5f_{20}\frac{H(q^4)G(q)}{G(q^4)H(q)}+2q\frac{f_1f_{10}f_{40}}{f_2}\left(\frac{G(q^2)}{H(q^2)}+q^2\frac{H^4(q^2)}{G^4(q^2)}\right)\notag\\
     &\equiv 2qf_5f_{20}\frac{H(q^4)G(q)}{G(q^4)H(q)}+2q\frac{f_1f_{10}f_{40}}{f_2H(q^2)G(q^8)}\left(G(q^2)G(q^8)+q^2G(q^2)G(q^8)\right)\notag\\
&\equiv 2qf_1f_4G(q^2)H(q^8)+2q\frac{f_1f_{10}f_{40}}{f_2f_4H(q^2)G(q^8)}\notag\\     
&\equiv 2qf_1f_4G(q^2)H(q^8)+2qf_1f_4G(q^2)H(q^8)\notag\\ 
&\equiv0\pmod{4}.
      \end{align*}
          
     \vskip .4cm \noindent \emph{Case $r=4$}. For this case, we extract the terms involving $q^{5n+4}$ from both sides of \eqref{c7} and then proceed as in the case for $r=1$ to deduce that     
    \begin{align}  \label{case4a} &\sum_{n=0}^{\infty}DSOME(25n+21)q^n\equiv  2f_5f_{20}\left(T(q)+\frac{q}{T(q^4)}\right)+2q\frac{f_1f_2^2f_{20}}{T(q^4)}+2\frac{f_1f_{10}^5}{f_2}T^3(q^2)\\
      &\equiv  \frac{2f_5f_{20}}{H(q)G(q^4)}\left(G(q)G(q^4)+qH(q)H(q^4)\right)+2qf_1f_4f_{20}\dfrac{H(q^4)}{G(q^4)}\notag\\
      &\quad+2\frac{f_1f_{10}f_{40}G(q^2)G(q^4)}{f_2H(q^2)H(q^4)}\notag\\
      &\equiv  2\frac{f_5f_{20}}{f_2H(q)G(q^4)}+2qf_1f_4^2H(q^8)+2f_1f_8H(q^4)G(q^{16})\notag\\
      &\equiv 2f_1f_2G(q)H(q^4)+2f_1f_8H(q^4)\left(G(q^{16})+qH(q^4)\right)\pmod{4}.\notag
       \end{align}
     
     Now, by \eqref{c5}, we have
     \begin{align}\label{Gmod2}
     G(q)\equiv \dfrac{f_8}{f_2}\left(G(q^{16})+qH(q^4)\right)\pmod{4}.
     \end{align}
     Using \eqref{Gmod2} in \eqref{case4a}, we arrive at
    \begin{align}
      \sum_{n=0}^{\infty}DSOME(25n+21)q^n
      &\equiv 4f_1f_2G(q)H(q^4)\equiv  0\pmod{4}.\notag
       \end{align}
     This completes the proof.  
  \end{proof}
  \begin{proof}[Proof of Theorem \ref{Theorem1.3}]
  Employing \eqref{c1} and \eqref{c2} in \eqref{c8}, and then extracting the terms involving $q^{5n}$ from the resulting identity, we find that
\begin{align*}
    &\sum_{n=0}^{\infty}DSOME(25n+1)q^n\\
    &=\frac{f_2}{8f_1}-\frac{f_5^7}{8f_{10}}+f_5^3f_{10}\bigg(x^{3}y+\frac{q^2}{x^{3}y}+3q\left(\frac{x^2}{y}-\frac{y}{x^2}\right)-5q\bigg)\\
    &\quad+\frac{1}{2}\bigg(\frac{f_5f_{10}^5}{f_1^6}
    +\frac{f_1^2f_5f_{10}^5}{f_2^4}\bigg)\bigg(2q\big(xy^2-\frac{q^2}{xy^2}\big)
    -q\big(\frac{y^3}{x}+q^2\frac{x}{y^3}\big)-5q^2\bigg)\\
   &\quad+2q\frac{f_2^3f_{10}^2}{f_1}+2q\frac{f_2^9f_5^{10}}{f_1^{19}}\bigg(10\left(x^5-\frac{q^2}{x^5}\right)+15q\bigg),\end{align*}
   where, as defined in Lemma \ref{lemmaxy},  $x=T(q)$ and $y=T(q^2)$. Employing \eqref{Lemma3a}--\eqref{x51byx5} and \eqref{qq12} in the above, we find that
\begin{align*}
&\sum_{n=0}^{\infty}DSOME(25n+1)q^n\notag\\
    &=\frac{f_2}{8f_1}-\frac{f_1^7}{8f_2^3}+f_5^3f_{10}\bigg(K+\frac{16q^2}{K}-3q\bigg)+
    \frac{q}{2}\bigg(\frac{f_5f_{10}^5}{f_1^6}
    +\frac{f_1^2f_5f_{10}^5}{f_2^4}\bigg)\bigg(K-\frac{4q^2}{K}-3q\bigg)\\
    &\quad+2q\frac{f_2^3f_{10}^2}{f_1}+10q\frac{f_2^9f_5^{10}}{f_1^{19}}\bigg(2\frac{f_1^6}{f_5^6}+25q\bigg)\\
     &=\frac{f_2}{8f_1}-\frac{f_1^7}{8f_2^3}+\frac{f_2^2f_5^8}{f_1f_{10}^4}+16q^2\frac{f_1f_{10}^6}{f_2f_5^2}-3qf_5^3f_{10}
    +\frac{q}{2}\frac{f_1^4f_{10}^3}{f_2^2f_5}+\frac{q}{2}\frac{f_2^2f_{10}^3}{f_1^4f_5}\\
    &\quad+2q\frac{f_2^3f_{10}^2}{f_1}+20q\frac{f_2^9f_5^4}{f_1^{13}}+250q^2\frac{f_2^9f_5^{10}}{f_1^{19}}.
\end{align*}
We now take congruence modulo 8 on both sides of the above and repeatedly use \eqref{Lemma4a} for further simplification. Accordingly, we obtain

\begin{align*}  &\sum_{n=0}^{\infty}DSOME(25n+1)q^n\notag\\
   &\equiv  \frac{9f_2}{8f_1}-\frac{f_1^7}{8f_2^3}-3qf_5^3f_{10}+\frac{q}{2}\frac{f_1^4f_{10}^3}{f_2^2f_5}+\frac{q}{2}\frac{f_2^2f_{10}^3}{f_1^4f_5}+6q\frac{f_2^3f_{10}^2}{f_1}+2q^2\frac{f_1f_5^2f_{20}^2}{f_2}\notag\\
   &\equiv \frac{9}{8}\frac{f_2}{f_1}-\frac{1}{8}\bigg(\frac{f_1^3f_5^4}{f_2f_{10}^{2}}-4q\frac{f_1^4f_{10}^3}{f_2^2f_5}\bigg)-3qf_5^3f_{10}+\frac{q}{2}\frac{f_1^4f_{10}^3}{f_2^2f_5}+\frac{q}{2}\bigg(\frac{f_{10}^5}{f_5^5}+4q\frac{f_2f_{10}^8}{f_1^3f_5^6}\bigg)
   \notag\\
   &\quad+6qf_1^3f_2f_{10}^2+2q^2\frac{f_1f_5^2f_{20}^2}{f_2}\notag\\
   &\equiv \frac{9}{8}\frac{f_2}{f_1}-\frac{1}{8}\frac{f_1^3f_5^4}{f_2f_{10}^2}-3qf_5^3f_{10}+q\frac{f_1^4f_{10}^3}{f_2^2f_5}+\frac{q}{2}\frac{f_{10}^5}{f_5^5}+6qf_1^3f_2f_{10}^2+4q^2\frac{f_1f_5^2f_{20}^2}{f_2}\notag\\
   &\equiv \frac{9}{8}\frac{f_2}{f_1}-\frac{1}{8}\bigg(\frac{f_2f_5^8}{f_1f_{10}^4}-4qf_5^3f_{10}\bigg)-3qf_5^3f_{10}+q\left(f_5^3f_{10}-4q\frac{f_1f_{10}^6}{f_2f_5^2} \right)+\frac{q}{2}\frac{f_{10}^5}{f_5^5}
   \notag\\
   &\quad+6qf_1^3f_2f_{10}^2+4q^2\frac{f_1f_5^2f_{20}}{f_2}\end{align*}\begin{align*}
   &\equiv\frac{f_2}{f_1}\left(\frac{9}{8}-\frac{f_5^8}{8f_{10}^4}\right)-\frac{3q}{2}f_5^3f_{10}+\frac{q}{2}\frac{f_{10}^5}{f_5^5}-2qf_1^3f_2f_{10}^2\pmod{8}.
\end{align*}
 With the use of  \eqref{c1} and \eqref{c2}, we rewrite the above as
 \begin{align*}
&\sum_{n=0}^{\infty}DSOME(25n+1)q^n\\
&\equiv \frac{f_{50}f_{25}^5}{f_5^6}\bigg(\frac{9}{8}-\frac{f_5^8}{8f_{10}^4}\bigg)\bigg(T(q^{10})-q^2-\frac{q^4}{T(q^{10})}\bigg)\bigg(T^4(q^5)+qT^3(q^5)+2q^2T^2(q^5)\\
&\quad+3q^3T(q^5)+5q^4-\frac{3q^5}{T(q^5)}
    +\frac{2q^6}{T^2(q^5)}-\frac{q^7}{T^3(q^5)}+\frac{q^8}{T^4(q^5)}\bigg)-\frac{3q}{2}f_5^3f_{10}+\frac{q}{2}\frac{f_{10}^5}{f_5^5}\\
    &\quad -2qf_{10}^2f_{25}^3f_{50}\bigg(T(q^{10})-q^2-\frac{q^4}{T(q^{10})}\bigg)\bigg(T^3(q^5)-3qT^2(q^5)+5q^3-\frac{3q^5}{T^2(q^5)}\\
    &\quad-\frac{q^6}{T^3(q^5)}\bigg).
    \end{align*}
    Extracting the terms  involving $q^{5n+1}$ from both sides of the above and then using \eqref{Lemma3b}, \eqref{Lemma3d}, \eqref{Lemma4b} and \eqref{qq12}, we deduce that
\begin{align*} &\sum_{n=0}^{\infty}DSOME(125n+26)q^n\\
&\equiv \frac{f_{10}f_5^5}{f_1^6}\bigg(\frac{9}{8}-\frac{f_1^8}{8f_2^4}\bigg)\bigg(x^3y+\frac{q^2}{x^3y}-2\left(\frac{x^2}{y}-\frac{y}{x^2}\right)-5q\bigg)
-\frac{3}{2}f_1^3f_2+\frac{1}{2}\frac{f_2^5}{f_1^5}
     \notag\\    &\quad-2qf_2^2f_5^3f_{10}\bigg(x^3y+\frac{q^2}{x^3y}+3q\left(\frac{x^2}{y}-\frac{y}{x^2}\right)-5q\bigg)\\
     &\equiv \frac{f_{10}f_5^5}{f_1^6}\bigg(\frac{9}{8}-\frac{f_1^8}{8f_2^4}\bigg)\left(K^2-\dfrac{4q^2}{K}-3q\right)
-\frac{3}{2}f_1^3f_2+\frac{1}{2}\frac{f_2^5}{f_1^5}-2f_2^2f_5^3f_{10}(K+q)\\
&\equiv \frac{9}{8}\frac{f_2^2f_5^3}{f_1^4f_{10}}-\frac{1}{8}\frac{f_1^4f_5^3}{f_2^2f_{10}}-\frac{3}{2}f_1^3f_2+\frac{f_2^5}{2f_1^5}-2f_1^2f_2^4\frac{f_5}{f_{10}}.
     \end{align*}
    Now by repeated use of \eqref{Lemma3a} and \eqref{Lemma3b} in the above, we find that 
     \begin{align}
     \label{mod8gen}&\sum_{n=0}^{\infty}DSOME(125n+26)q^n\\   
     &\equiv\frac{9}{8}\left(\frac{f_{10}}{f_5}+4q\frac{f_2f_{10}^4}{f_1^3f_5^2}\right)-\frac{1}{8}\left(\frac{f_5^7}{f_{10}^3}-4q\frac{f_1f_5^2f_{10}^2}{f_2}\right)-\frac{3}{2}\left(\frac{f_2^3f_5^4}{f_1f_{10}^2}-4q\frac{f_2^2f_{10}^3}{f_5}\right)\notag\\
     &\quad+\frac{1}{2}\left(\frac{f_2^3f_{10}^2}{f_1f_5^4}+4q\frac{f_2^4f_{10}^5}{f_1^4f_5^5}\right)-2\left(\frac{f_1^3f_2f_5^4}{f_{10}^2}+q\frac{f_1^4f_{10}^3}{f_5}\right)\notag\\
       &\equiv\frac{9}{8}\frac{f_{10}}{f_5}+\frac{9q}{2}\left(\frac{f_1f_{10}^6}{f_2f_5^6}+4q\frac{f_{10}^9}{f_1^2f_5^7}\right)-\frac{1}{8}\frac{f_5^7}{f_{10}^3}+\frac{q}{2}\frac{f_1f_5^2f_{10}^2}{f_2}\notag\\
         &\quad-\frac{3}{2}\left(\frac{f_5^7}{f_{10}^3}+q\frac{f_1f_5^2f_{10}^2}{f_2}\right)+\frac{f_2^3f_{10}^2}{2f_1f_5^4}-2f_1^3f_2-2q\left(\frac{f_1^5f_5^2f_{10}^2}{f_2^3}+q\frac{f_1^6f_{10}^7}{f_2^4f_5^3}\right)\notag\\
         &\equiv\frac{9}{8}\frac{f_{10}}{f_5}+\frac{9q}{2}\frac{f_1f_{10}^6}{f_2f_5^6}-\frac{13}{8}\frac{f_5^7}{f_{10}^3}-3q\frac{f_1f_5^2f_{10}^2}{f_2}+\frac{1}{2}\left(\frac{f_{10}}{f_5}+q\frac{f_1f_{10}^6}{f_2f_5^6}\right)\notag\\
         &\quad-2f_1^3f_2
         \notag\end{align}\begin{align}
        &\equiv\frac{13}{8}\left(\frac{f_{10}}{f_5}-\frac{f_5^7}{f_{10}^3}\right)+2\left(\frac{f_2^3f_5^4}{f_1f_{10}^2}-\frac{f_5^7}{f_{10}^3}\right)-2f_1^3f_2\notag\\
            &\equiv\frac{13}{8}\left(\frac{f_{10}}{f_5}-\frac{f_5^7}{f_{10}^3}\right)-2\frac{f_{10}}{f_5}\pmod{8}.\notag
\end{align}
It follows from the above that
\begin{align*}    DSOME(625n+125r+26)\equiv0\pmod{8},\, 1\leq r\leq4.
\end{align*}
This completes the proof of Theorem \ref{Theorem1.3}.
  \end{proof}
 \begin{proof}[Proof of Theorem \ref{Theorem1.4}] Using \eqref{DSOMEgf} in \eqref{mod8gen}, we have
 \begin{align}     
\sum_{n=0}^{\infty}DSOME(125n+26)q^n&\equiv  5\sum_{n=0}^{\infty}DSOME(n)q^{5n}-2\frac{f_{10}}{f_5}\pmod{8},\label{eq123}
\end{align}
which yields
     \begin{align*}
        \sum_{n=0}^{\infty} DSOME(625n+26)q^n\equiv  5\sum_{n=0}^{\infty} DSOME(n)q^n-2\frac{f_2}{f_1}\pmod{8}.
        \end{align*}
   Employing \eqref{c1} and \eqref{c2} in the above and then extracting the terms involving $q^{5n+1}$ from the resulting congruence, we find that
\begin{align*}
&\sum_{n\geq0}DSOME\bigg(3125n+651\bigg)q^n\notag\\
&\equiv 5\sum_{n\geq0}DSOME\bigg(5n+1\bigg)q^n-2\frac{f_5^5f_{10}}{f_1^6}\left(K^2-\dfrac{4q^2}{K}-3q\right),
        \end{align*}
    which, by \eqref{qq12}, implies that
\begin{align*}
\sum_{n\geq0}DSOME\bigg(3125n+651\bigg)q^n&\equiv 5\sum_{n\geq0}DSOME\bigg(5n+1\bigg)q^n-2\frac{f_2^2f_5^3}{f_1^4f_{10}}\notag\\
&\equiv 5\sum_{n\geq0}DSOME\bigg(5n+1\bigg)q^n-2\frac{f_{10}}{f_5}\pmod{8}.
        \end{align*}
It follows from the above congruence and \eqref{eq123}  that
          \begin{align}\label{y1}       
        &\sum_{n\geq0}DSOME\bigg(3125n+651\bigg)q^n\\
&\equiv 5\sum_{n\geq0}DSOME\bigg(5n+1\bigg)q^n+\sum_{n\geq0}DSOME\notag
\bigg(125n
+26\bigg)q^n\notag\\&\quad-5\sum_{n\geq0}DSOME(n)q^{5n}\pmod{8}.\notag
     \end{align}
Comparing the coefficients of $q^{5n}$ as well as $q^{5n+r}$, where  $1\leq r\leq4$, in turn,    on both sides of \eqref{y1}, we obtain
  \begin{align*} DSOME(15625n+651)&\equiv 5DSOME(25n+1)+DSOME(625n+26)\\ &\quad-5DSOME(n)\pmod{8}\\\intertext{and}
 DSOME(15625n+3125r+651)&\equiv 5DSOME(25n+5r+1)\\
&\quad+DSOME(625n
+125r+26)\pmod{8}\notag.
\end{align*}
Thus we  complete the proof of Theorem \ref{Theorem1.4}.
  \end{proof}
\section{conclusion and remarks}
In this paper, we have expressed the generating function of $DSOME(n)$ found by Andrews and Ghosh Dastidar in a closed form. The new expression allows us to find new congruences modulo 4 and 8. It seems that far more congruences are yet to be discovered. We propose the following conjecture based on numerical calculations.
\begin{conjecture} For all $n\geq0$,
    \begin{align*}
        DSOME(50n+21)&\equiv0\pmod{8}\\\intertext{and}
        DSOME(100n+71)&\equiv0\pmod{16}.
    \end{align*}
\end{conjecture}

\subsection*{Acknowledgements}
The authors wish to thank the anonymous referee and Dazhao Tang for their helpful comments. The second author was partially supported by University Grants Commission, Government of India, under the UGC-JRF scheme (Ref. No. 221610056019). The author thanks the funding agency.


\end{document}